\documentclass[10pt]{article}
\usepackage{amsmath}
\usepackage{amsfonts, amssymb, amsthm}
\usepackage{epsfig}
\usepackage{color}
\usepackage{multicol}

\newlength{\hchng}
\newlength{\vchng}
\newtheorem{thm}{Theorem}[section]
\newtheorem{prop}[thm]{Proposition}
\newtheorem{cor}[thm]{Corollary}

\newtheorem{lemma}[thm]{Lemma}

\newtheorem{preremark}[thm]{Remark}
\newenvironment{remark}{\begin{preremark}\rm}{\medskip \end{preremark}}
\numberwithin{equation}{section}

\newcommand{\norm}[1]{\left\Vert#1\right\Vert}

\newcommand{\R}{\mathbb R}
\newcommand{\eps}{\varepsilon}

\newcommand{\grad} {\nabla}

\newcommand{\bdary} {\partial}

\newcommand{\dd} {\; \mathrm{d}}

\DeclareMathOperator*{\osc}{osc}

\DeclareMathOperator{\tr}{Tr}

\newcommand{\Q}{Q}

\def\ds{\displaystyle}

\title{H\"older continuity to Hamilton-Jacobi equations with superquadratic growth in the gradient and unbounded right-hand side}
\author{Pierre Cardaliaguet\thanks{CEREMADE (UMR CNRS no. 7534), Universit\'e Paris-Dauphine; e-mail: cardaliaguet@ceremade.dauphine.fr} and Luis Silvestre\thanks{University of Chicago; email:luis@math.uchicago.edu}}

\begin{document}
\maketitle

\begin{abstract} We show that solutions of time-dependent degenerate parabolic equations with super-quadratic growth in the gradient variable and possibly unbounded right-hand side are locally ${\mathcal C}^{0,\alpha}$. Unlike the existing (and more involved) proofs for equations with bounded right-hand side, our arguments rely on constructions of sub- and supersolutions combined with improvement of oscillation techniques.  
\end{abstract}

%
%

\section{Introduction}

We study the regularization effect of Hamilton-Jacobi equations with rough coefficients
\begin{equation}\label{eq:HJ}
u_t + H(x,t,D u,D^2 u) = 0.
\end{equation}
Our aim is to show that, if the growth of $H$ with respect to the gradient variable is sufficiently large, then bounded solutions will be H\"{o}lder continuous, with a H\"older modulus of continuity independent of the regularity of $H$ with respect to $(x,t)$. The regularization effect is based only on the strong coercivity assumption of $H$ respect to $Du$, and not on the second order diffusion, which can be degenerate or even nonexistent.

We prove three results. The first one is the simplest case in which $H$ does not depend on $D^2 u$ (there is no diffusion).
\begin{thm} \label{t:intro1}
Let $u$ be a bounded viscosity solution of \eqref{eq:HJ} in $B_1 \times [0,1]$ for a function $H=H(x,t,Du)$ that does not depend on $D^2u$. Assume that for all $x$ and $t$, \[\frac 1A |Du|^p - A \leq H(x,t,Du) \leq A|Du|^p + A,\] for some constants $A$ and $p>1$. Then $u$ is $C^\alpha$ in $B_{1/2} \times [1/2,1]$, where $\alpha$ and the $C^\alpha$ norm of $u$ depend on $p$, $A$ and $\norm{u}_{L^\infty}$ only.
\end{thm}

This result was already proved in \cite{ca08} and \cite{CC} using representation formulas for Hamilton-Jacobi equations and the reverse H\"older inequality. The proof that we present here is shorter and more elementary. Our new proof does not rely on the path representation of the Hamilton-Jacobi equation. Instead, an improvement of oscillation at all scales is obtained by comparing the solution to explicit super and subsolutions given by the classical formula of Lax and Oleinik. 

Using the same idea, we can also address the case of Hamilton-Jacobi equations with diffusion. The proof is slightly longer because we cannot use the formula of Lax and Oleinik and instead we need to construct explicit super and subsolutions.

In order to handle a large class of diffusions, possibly degenerate, we define the extremal operators $m^\pm(X)$ for any symmetric matrix $X$. Let $m^+(X)$ be the maximum positive eigenvalue of $X$ (zero if none) and $m^-(X)$ be the minimum negative eigenvalue of $X$ (zero if none).

\begin{thm} \label{t:intro2}
Let $u$ be a bounded viscosity solution of \eqref{eq:HJ} in $B_1 \times [0,1]$. Assume that for all $x$ and $t$, \[\frac 1A |Du|^p - Am^+(D^2u) - A \leq H(x,t,Du,D^2 u) \leq A|Du|^p - Am^-(D^2u) + A,\] for some constants $A$ and $p>2$. Then $u$ is $C^\alpha$ in $B_{1/2} \times [1/2,1]$, where $\alpha$ and the $C^\alpha$ norm of $u$ depend on $p$, $A$ and $\norm{u}_{L^\infty}$ only.
\end{thm}

Note that the assumption on $H$ allows us to apply Theorem \ref{t:intro2} to Hamilton-Jacobi equations with degenerate diffusion. In this result, the regularization effect is a consequence of the growth of $H$ respect to $Du$ only. The diffusion plays no role in the proof and in fact since $p$ is assumed to be larger than two, it can be treated as a negligible effect in small scales.

Theorem \ref{t:intro2} has already been proved in \cite{CC} and \cite{cr11} using a stochastic representation of the equation and the reverse H\"older inequality. Our proof relies on comparison principle only with explicit super and subsolutions. No techniques from probability are used.

Our third result is new and allows us to handle equations with a right hand side that is only bounded below but possibly unbounded above. In order to be able to apply this result we need stronger assumptions on the diffusion. The diffusion can still be degenerate (or even non existent), but we need it to depend on $x$ smoothly.
\begin{thm} \label{t:intro3}
Let $u$ be a bounded viscosity solution of \eqref{eq:HJ} in $B_1 \times [0,1]$. Assume that for all $x$ and $t$, 
\[\frac 1A |Du|^p - \tr(B(x,t) D^2u) - f(x,t) \leq H(x,t,Du,D^2 u) \leq A|Du|^p - Am^-(D^2u) + A,\] 
where $A\geq 1$, $p>2$, $f \in L^m$ for some $m > 1+d/p$ ($d$ being the space dimension) and $B(x,t)$ is a $d \times d$ symmetric nonnegative matrix that is $C^1$ respect to $x$.

Then $u$ is $C^\alpha$ in $B_{1/2} \times [1/2,1]$, where $\alpha$ and the $C^\alpha$ norm of $u$ depend on $p$, $A$, $\norm{u}_{L^\infty}$, $||B||_{C^1}$ and $||f||_{L^m}$ only.
\end{thm}

Note that in this setting we can allow unboundedness of $H$ only  from one side. This restriction seems mandatory to ensure the boundedness of the solution.

This extension to unbounded coefficient is new, although related results, for elliptic problems, can be found in \cite{DaPo}. 
It is partially motivated by  \cite{LL07mf}, in which solutions of the so-called mean field game equations are related with some optimal control problems for Hamilton-Jacobi equations of the form
\[
u_t -\Delta u +  H(x,t,D u) = f(x,t)
\]
where $f\geq 0$ is a distributed  control which has to be bounded in some Lebesgue space. 

H\"{o}lder regularity of Hamilton-Jacobi equation with a large growth with respect to the gradient variable has been addressed in several papers: in \cite{CDLP}, \cite{Ba09} it is proved that subsolutions of stationary equations are H\"{o}lder continuous (even up to the boundary of the domain) as soon as the Hamiltonian is superquadratic. The H\"{o}lder regularity of solutions of \eqref{eq:HJ} as in Theorem \ref{t:intro1} or \ref{t:intro2} has already been obtained in \cite{cr11}. This paper, based on techniques first introduced in  \cite{CC}, \cite{ca08}, relies on (rather heavy) stochastic control representation for solutions of \eqref{eq:HJ} combined with reverse H\"{o}lder inequalities. On the contrary, the new proof we present here  uses only comparison principle with explicit functions that we can construct. There is no component of probability in the proof. 

For Hamilton-Jacobi equations with second order diffusion, if the growth with respect to the gradient variable is at most quadratic, then the second order diffusion dominates in small scales. For uniformly parabolic equations in non divergence form, the solutions are H\"older continuous. In fact, by a simple exponential change of variables (first observed in \cite{MR1048584} for elliptic equations), this statement is a simple consequence of the H\"older estimates from Krylov and Safonov for parabolic equations without gradient terms. Hamilton-Jacobi equations with fractional diffusions have also been studied in \cite{silvestre2011differentiability} and \cite{Si2} in the case where the diffusion is stronger than the effect of the gradient on the Hamiltonian. Our assumption $p>2$ in Theorems \ref{t:intro2} and \ref{t:intro3} explores the complementary condition in which the second order diffusion is weaker than the gradient terms at small scales. In our results, we treat the diffusion as a perturbation of the equation that plays no role in the regularization of the solutions. The case $1<p\leq 2$ with degenerate diffusion is unclear, and we do not know if the solutions would become regular in that case.

The proofs we present in this article are based on the iteration at all scales of an improvement of oscillation lemma. This lemma says that the oscillation of a solution $u$ in a cylinder is smaller by a fixed factor than the oscillation in a larger cylinder (also comparable by a fixed factor). The proof of this oscillation Lemma has the interesting feature that it uses the fact that the Hamilton-Jacobi equation \emph{sees points} (i.e. points have positive capacity, i.e. controlled processes can be forced to hit a given point). The idea is that if $u$ is small at one single point on the bottom of a parabolic cylinder, then the equation forces $u$ to be small everywhere inside the cylinder. In the opposite case, if $u$ is large at every point at the bottom of the cylinder, then naturally $u$ will be large everywhere inside. Both implications are proved comparing $u$ with explicit super and subsolutions.

The organization of the paper is the following. After this introduction, we have a small preliminary section where we analyze the scaling of the equation and we outline a general iterative improvement of oscillation scheme to prove H\"older continuity. Then we have three sections where we prove theorems \ref{t:intro1}, \ref{t:intro2} and \ref{t:intro3} respectively. Even though there is a significant overlap between these results, we kept the three sections independent except for one lemma used to prove Theorem \ref{t:intro2} that is also used in the proof of Theorem \ref{t:intro3}. A reader interested in Theorem \ref{t:intro1} only, needs to read sections \ref{s:prelim} and \ref{sec:order1} only (which is simple and very short). A reader interested in Theorem \ref{t:intro2}, needs to read sections \ref{s:prelim} and \ref{sec:order2} only.

\section{Preliminaries}
\label{s:prelim}

\subsection{Scaling}
\label{s:scaling}
We start this section we exploring the effect of several different scalings of the equation.

Let $u$ be a sub or super solution to
\[ u_t + A|D u|^p - \eps m^{\pm} (D^2 u) = f \]
then the function $v(x,t) = cu(ax,bt)$ is a sub(super) solution of the equation
\[ v_t + A a^{-p} b c^{1-p} |D v|^p - \eps a^{-2} b m^{\pm} (D^2 u) = bc f(ax,bt) \]

In order to prove a ${\mathcal C}^\alpha$ estimate, we need scalings that preserve the constant $A$ in front of the $|Du|^p$ and diminishes the effect of diffusion and zero order term. 

A first choice to consider is $c=1$, $b=a^p$ with $a<1$. For $a<1$, this scaling corresponds to zooming without altering the values of $v$ ($L^\infty$ scaling). If $\eps=0$ and $f=0$ (as needed for Theorem \ref{t:intro1}), the scaling preserves the equation and we do not need to check anything else. 

If $\eps>0$, we would want the scaling not to enhance the second order term. Thus we would need $p>2$, so that $a^{-2} b<1$ when $a<1$. This will be the case in the proof of Theorem \ref{t:intro2}.

When the right hand side $f$ is unbounded, we need to check the effect of scaling on its $L^m$ norm. We have
\[
|b f(a\cdot,b\cdot)|_{L^m}= a^{p(1-1/m)-d/m} \|f\|_m.
\]
Thus, one has $a^{p(1-1/m)-d/m} <1$ as soon as  $p(m-1)>d$. This is the case for Theorem \ref{t:intro3}.

Let us now look at the effect of a ${\mathcal C}^\alpha$ scaling. Let $\alpha \in (0,1)$ and $\beta = p - \alpha (p-1)$. Then the scaled function $u_r(x,t) = r^{-\alpha} u(rx,r^\beta t)$ is a sub(super) solution of
\[ \partial_t (u_r) + A |D u_r|^p - \eps r^{-\alpha (p-1) + p - 2} m^{\pm} (D^2 u) = r^{p(1-\alpha)} f(rx,r^\beta t). \]
If $\eps=0$ and $f=0$, then the equation is preserved as before. When $\eps>0$ or $f$ is unbounded, we need to choose $\alpha$ small enough so that the effect of the diffusion and the right hand side decrease as we scale in. For the diffusion part we need $r^{-\alpha (p-1) + p - 2} < 1$. Regarding the $L^m$ norm of $f$, 
\[
|r^{p(1-\alpha)} f(r\cdot,r^\beta \cdot) |_{L^m}= r^{\delta}\|f\|_m\qquad {\rm with}\; \delta = \frac{1}{m}(p(m-1)-d)+\frac{\alpha}{m}((m-1) p+1)
\]
and, if $p(m-1)>d$, one can choose again $\alpha$ small in order to diminish the effect of the right-hand side. 

\subsection{Iterative improvement of oscillation}

One of the most common methods to prove an interior H\"older regularity result is by proving an improvement of oscillation Lemma that is iterated at all scales. The improvement of oscillation Lemma says that the oscillation of a solution to a given equation in a parabolic cylinder decreases by a fixed factor (less than one) when the size of cylinder is reduced by another given factor. The iteration of such result in a decreasing sequence of cylinders produces a H\"older modulus of continuity. We work out this general iterative procedure in the following Proposition.

\begin{prop} \label{p:scaleiteration}
Let $Q_r(x,t)$ denote the cylinder $B_r(x) \times [t-r^\beta,t]$ for some constant $\beta>0$. Let $u$ be a function with oscillation $1$ in $Q_2(0,0)$ such that the following ``improvement of oscillation'' property holds:
there exists $\lambda \in (0,1)$  and $\alpha\in(0,1)$ such that for all $0<r\leq 1$ and any cylinder $Q_r(x,r)$ contained in $Q_2(0,0)$ we have the following property
\begin{equation} \label{e:improv-of-oscillation}
\text{if } \osc_{Q_r(x,t)} u \leq r^\alpha, \; \text{then}\; \osc_{Q_{\lambda r}(x,t)} u \leq (\lambda r)^\alpha.
\end{equation}
Then $u$ is  $C^\alpha(Q_{1}(0,0))$ and 
$$
|u(x,s)-u(y,t)|\leq C \left[ |x-y|^\alpha+|s-t|^{\frac{\alpha}{\beta}}\right]\qquad \forall (x,s),(y,t)\in Q_1\;,
$$
where  $C$ depends on $\lambda$  and $\alpha$ only.
\end{prop}

Typically, one proves that solutions to an equation are H\"older continuous by showing that \eqref{e:improv-of-oscillation} holds at unit scale $r=1$. Then it holds for all $r>0$ using the scale invariance of the equation. 
The proof of Proposition \ref{p:scaleiteration} is fairly standard. We include it here for completeness.

\begin{proof}
Let $(x,t) \in Q_{1}(0,0)$. In order to show that $u \in C^\alpha(Q_{1}(0,0))$, we must show that $\osc_{Q_{r}(x,t)} u \leq C r^\alpha$ for all $r \in (0,1)$. We will prove by induction with $C=1$ and the sequence of radii $r_k = \lambda^k $. Then it follows for all $r$ by using $C=\lambda^{-\alpha}$.

For $k=0$, we trivially know that $\osc_{Q_{1}(x,t)} u \leq 1$. Therefore, the desired inequality holds for $r_0$. We now proceed with the inductive process. Assume that we know that $\osc_{Q_{r_k}(x,t)} u \leq r_k^\alpha$ up to a certain $k \geq 0$. But then the hypothesis \eqref{e:improv-of-oscillation} is exactly what we need to obtain \[ \osc_{Q_{r_{k+1}}(x,t)} u = \osc_{Q_{\lambda r_k}(x,t)} u \leq (\lambda r_k)^{\alpha} = r_{k+1}^\alpha.\]
And we finish the proof.
\end{proof}

\section{First order equations}\label{sec:order1}

In this section prove Theorem \ref{t:intro1} about the H\"older regularity of first order Hamilton-Jacobi equations with superlinear growth in the gradient variable, but without dependence on $D^2 u$ and without an unbounded right hand side. 

Theorem \ref{t:intro1} is a consequence of the following more precise result.

\begin{thm}\label{t:main1}
Let $u$ be a continuous viscosity solution of the following two equations 
\begin{align}
u_t + A |D u|^p + A \geq 0  \label{e:fo1} \\
u_t + \frac 1 A |D u|^p - A\leq 0 \label{e:fo2} 
\end{align}
in $B_1\times (0,1)$, where $p>1$. Then 
$$
|u(x,t)-u(y,s)|\leq C \left[ |x-y|^\alpha+|t-s|^{\frac{\alpha}{p-\alpha(p-1)}}\right] \qquad \forall (x,t),(y,s)\in B_{1/2}\times [1/2,1]
$$
for some $\alpha \in (0, p')$ (where $1/p+1/p'=1$) and $C$ depending on $A$, $p$, $\|u\|_\infty$ and dimension. 
\end{thm}

\begin{remark} Following section \ref{sec:unbounded},  the conclusion remains unchanged if $u$ is a continuous viscosity solution of the two equations
\begin{align}
u_t + A |D u|^p + A \geq 0\\
u_t + \frac 1 A |D u|^p - f(x,t)\leq 0
\end{align}
where $p>1$ and $f\in L^m(B_1\times(0,1))$ for some $m$ with $p(m-1)>d$. In this case, the constants $\alpha$ and $C$ depend on 
 $A$, $p$, $\|f\|_{L^m(B_1\times (0,1))}$, $\|u\|_\infty$ and dimension. 
\end{remark}

The proof of the improvement of oscillation lemma uses the Lax-Oleinik formula to construct local approximations of the function $u$ and then apply comparison principle. 

The classical Lax-Oleinik formula says that the solution $v$ of the equation
\[ v_t + H(D v)=0 \text{  in  } B_R \times [-T,0] \]
is given by the formula
\[ v(x,t) = \min_{(y,s) \in \bdary (B_R \times [-T,0])} v(y,s) + (t-s) L \left( \frac{x-y}{t-s} \right). \]
where $L$ is the Legendre transform of $H$: $L(\xi) = \max r\xi - H(r)$ and $\bdary (B_R \times [-T,0])$ denotes the parabolic boundary $B_R \times \{-T\} \cup \bdary B_R \times [-T,0]$. So, the formula above allows us to compute the values of $v$ everywhere from the values on the parabolic boundary only.

In particular, if $H(\xi) = \eps + A|\xi|^p$, $L(r) = -\eps + c_p A^{1-p'} |r|^{p'}$, where $1/p+1/p'=1$ and $c_p$ is some constant depending on $p$. For $H(\xi) = -\eps + \frac 1A|\xi|^p$, $L(r) = \eps + c_p A^{p'-1} |r|^{p'}$

\begin{lemma}[improvement of oscillation] \label{l:diminish-of-oscillation}
There exist positive constants $T$ (large), $\theta$ and $\eps$ (small) such that for all continuous functions $u$ which satisfy
\begin{align}
u_t + A |D u|^p + \eps &\geq 0 && \text{in } B_2 \times [-T,0] \label{e:1} \\
u_t + \frac 1A |D u|^p - \eps &\leq 0 && \text{in } B_2 \times [-T,0] \label{e:2} \\
0 \leq u \leq 1 &&& \text{in } B_2 \times [-T,0]
\end{align}
then
\[ \osc_{\Q_1 = B_1 \times [-1,0]} u \leq (1-\theta) \]
\end{lemma}

\begin{proof}
Recall that the Legendre transform of the function $\xi \mapsto A |\xi|^p$ is given by the formula $c_p A^{1-p'} |\xi|^{p'}$.

We start by pointing out how $T$, $\theta$ and $\eps$ are chosen so as to make sure there is no circular dependence of the constants. We need the following three inequalities to hold:
\begin{align}
c_p (T-1)^{1-p'} A^{p'-1} 3^{p'} &\leq 1-4\theta \label{e:a1} \\
c_p T^{1-p'} A^{p'-1} &\geq 2\theta \label{e:a2} \\
\eps T &\leq \theta  \label{e:a3}
\end{align}

We can start by choosing $T$ large so that the left hand side of \eqref{e:a1} is strictly less than $1$. Then $\theta>0$ is chosen small enough to make \eqref{e:a1} and \eqref{e:a2} hold. Finally $\eps>0$ is chosen small so that \eqref{e:a3} holds.

Now we prove that the result of the lemma holds. Let $(x,t)$ be any point in $Q_1$. From \eqref{e:2} we obtain an upper bound for $u(x,t)$ using Lax Oleinik formula and the comparison principle
\begin{align*}
 u(x,t) &\leq \eps (t+T) + \min_{y\in B_2} u(y,-T) + c_p (T+t)^{1-p'} A^{p'-1} |x-y|^{p'} \\
& \leq \eps T + \min_{y \in B_2} u(y,-T) + c_p (T-1)^{1-p'} A^{p'-1} 3^{p'} \\
& \leq 1-3\theta + \min_{y \in B_2} u(y,-T)
\end{align*}

The last inequality above follows from \eqref{e:a1}. We see that if there is one point $y \in B_2$ such that $u(y,-T) \leq 2\theta$, then $u \leq 1-\theta$ in $Q_1$ and we would finish the proof. 

We must now analyze the case in which $u(y,-T) \geq 2\theta$ for all $y \in B_2$. For this case, we apply the Lax Oleinik formula and comparison principle again, but this time with \eqref{e:1}.
\[
u(x,t) \geq \min_{(y,s) \in \bdary (B_2 \times [-T,0])} -\eps (-s+t) + u(y,s) + c_p (-s+t)^{1-p'} A^{1-p'} |x-y|^{p'} \]

Recall that, by $\bdary (B_2 \times [-T,0])$ we denote the parabolic boundary $B_2 \times {-T} \cup \bdary B_2 \times [-T,0]$. We divide the estimate for $u(x,t)$ depending on what part of the boundary $(y,s)$ lies on.

If $y \in B_2$ and $s = -T$, 
\begin{align*}
-\eps (-s+t) + u(y,s) + c_p (-s+t)^{1-p'} A^{1-p'} |x-y|^{p'} &\geq - \eps T + \min_{y \in B_2} u(y,-T) \\
&\leq \theta
\end{align*}

On the other hand, if $y \in \bdary B_2$ and $s \in [-T,0]$ we use \eqref{e:a2} and $u \leq 1$,
\begin{align*}
-\eps (-s+t) + u(y,s) + c_p (-s+t)^{1-p'} A^{1-p'} |x-y|^{p'} &\geq \eps T  + c_p T^{1-p'} A^{1-p'} \\
&\leq \theta
\end{align*}

In any case we obtain $u(x,t) \geq \theta$ for all $(x,t) \in Q_1$. This finishes the second case of the proof.
\end{proof}

We now restate Lemma \ref{l:diminish-of-oscillation} in a more convenient form.

\begin{cor} \label{c:oscillation-1storder}
There exists a $\lambda \in (0,1)$ such that if  $u$ satisfies \eqref{e:fo1} and \eqref{e:fo2} in the unit cylinder $Q_1=B_1\times [-1,0]$ and $\osc_{Q_1} u \leq 1$, then $\osc_{B_\lambda \times [-\lambda^{p/2},0]} u \leq (1-\theta)$.
\end{cor}

\begin{proof}
Consider $\tilde u(x,r) = u(r^{-1}x,r^{-p}t) - \min_{Q_1} u$ where $r = \max(2,T^{1/p})$. The function $\tilde u$ satisfies
\begin{align*}
\tilde u_t + A|D \tilde u|^p + r^{-p} \eps &\geq 0, \\ 
\tilde u_t + \frac 1A |D \tilde u|^p - r^{-p} \eps &\leq 0,
\end{align*}
in the cylinder $B_r \times [-r^p,0]$ which includes $B_2 \times [-T,0]$. According to Lemma \ref{l:diminish-of-oscillation}, if $\eps$ is small enough and $\lambda = r^{-2}$,
\[ \osc_{B_\lambda \times [-\lambda^{p/2},0]} u \leq \osc_{B_1 \times [-1,0]} \tilde u \leq 1-\theta.\]
\end{proof}

Applying Corollary \ref{c:oscillation-1storder} at all scales with Proposition \ref{p:scaleiteration}, we prove Theorem \ref{t:main1}.

\begin{proof} [Proof of Theorem \ref{t:main1}.]
Without loss of generality, we prove the estimate for $(x,t) = (0,0)$ and $(y,s)$ sufficiently close to $(0,0)$. Also without loss of generality we assume that $0 \leq u \leq 1$. Otherwise we consider $(u - \inf u)/(\osc u)$ instead, which also solves inequalities with \eqref{e:fo1} and \eqref{e:fo2} with a slightly modified $A$ (see section \ref{s:scaling}).

We start by scaling the equation so that the zeroth order term is negligible and the oscillation of the function is one. We consider $\tilde u(x,t) = u(\rho x, \rho^p t)$ which solves
\begin{align*}
\tilde u_t + A|D \tilde u|^p + \rho^p A &\geq 0, \\ 
\tilde u_t + \frac 1A |D \tilde u|^p - \rho^p A &\leq 0,
\end{align*}
We choose $\rho$ small enough so that $\rho^p A < \eps$, for the constant $\eps$ of Corollary \ref{c:oscillation-1storder}. Note that this rescaling is just zooming in around a neighborhood of $(0,0)$. If we prove that $\tilde u$ is H\"older continuous at $(0,0)$, that is equivalent to $u$ being H\"older continuous at $(0,0)$ and the norms are related by a constant depending on $\rho$.

We now show that $\tilde u$ satisfies the iterative hypothesis of Proposition \ref{p:scaleiteration}. Let $\alpha\in(0, \frac{p}{2(p-1)})$ be a small positive constant that will be chosen below and $Q_r=B_r(0)\times[-r^{p-\alpha(p-1)}, 0]$. Suppose that $\osc_{Q_r} \tilde u \leq r^\alpha$.   Let $\tilde u^r(x,t) = r^{-\alpha} \tilde u(rx,r^{p-\alpha(p-1)}t)$. The function $\tilde u^r$ satisfies $\osc_{Q_1} \tilde u^r \leq 1$ and the inequalities
\begin{align*}
\tilde u^r_t + A|D \tilde u^r|^p + r^{p(1-\alpha)} \rho^p A &\geq 0, \\ 
\tilde u^r_t + \frac 1A |D \tilde u^r|^p - r^{p(1-\alpha)} \rho^p A &\leq 0.
\end{align*}
Note that $r^{p(1-\alpha)} \rho^p A \leq \rho^p A < \eps$. Thus, we can apply Corollary \ref{c:oscillation-1storder} and obtain $\osc_{B_\lambda \times [-\lambda^{p/2},0]} \tilde u^r \leq 1-\theta$. Let us choose $\alpha$ such that $\lambda^\alpha \geq 1-\theta$. Therefore we have
\[ \osc_{Q_{\lambda r}} \tilde u = r^\alpha \osc_{B_\lambda \times [-\lambda^{p-\alpha(p-1)},0]} \tilde u^r \leq  r^\alpha \osc_{B_\lambda \times [-\lambda^{p/2},0]} \tilde u^r \leq (1-\theta) r^\alpha \leq (\lambda r)^\alpha. \]
And thus, we proceed with Proposition \ref{p:scaleiteration}.
\end{proof}

\section{Second order equations with bounded right-hand side}
\label{sec:order2}

In this section we prove Theorem \ref{t:intro2} about Hamilton-Jacobi equations with possibly degenerate diffusions.

Theorem \ref{t:intro2} is a consequence of the following more precise statement.

\begin{thm}\label{t:main2}
Let $u$ be a continuous function that satisfies the following two inequalities in the viscosity sense:
\begin{align}
u_t + A|D u|^p + A m^-(D^2 u) + A &\geq 0 \qquad \text{and} \label{e:m1}\\
u_t + \frac 1 A |D u|^p + A m^+(D^2 u) - A &\leq 0 \label{e:m2}
\end{align}
in $B_1\times (-1,0)$. Then 
$$
|u(x,t)-u(y,s)|\leq C \left[ |x-y|^\alpha+|t-s|^{\frac{\alpha}{p-\alpha(p-1)}}\right] \qquad \forall (x,t),(y,s)\in B_{1/2}\times [-1/2,0]
$$
for some $\alpha \in (0, p')$ (where $1/p+1/p'=1$) and $C$ depending on $A$, $p$, $\|u\|_\infty$ and dimension. 
\end{thm}

The key to prove Theorem \ref{t:main2} is to prove an improvement of oscillation lemma. More precisely we will see that if the oscillation of a function $u$ in space-time cylinder is one and $u$ satisfies the two inequalities
\begin{align*}
u_t + A |D u|^p - \eps m^-(D^2 u) &\geq -\eps \\
u_t + \frac 1 A |Du|^p - \eps m^+(D^2 u) &\leq \eps
\end{align*}
then, if $\eps$ is small enough, the oscillation of $u$ in a smaller cylinder is less than $1-\theta$ for some $\theta>0$ depending only on $A$, $p$ and dimension.

The assumption of $\eps$ small is not restrictive, since any function $u$ satisfying \eqref{e:m1} and \eqref{e:m2} can be rescaled by $\tilde u(x,t) = u(rx, r^p t)$ in order to make the diffusion term and right hand side as small as we wish.

Along this section, we will find constants $R>0$, $r>0$, $\eps_0>0$ and $\theta>0$ depending only on $A$, $p$ and dimension, so that if $\eps \leq \eps_0$ and the oscillation of $u$ in $B_{R+r} \times [0,1]$ is less than one, then the oscillation of $u$ in $B_{R/2} \times [1/2,1]$ is less than $(1-\theta)$. Iterating this improvement of oscillation lemma at all scales leads to a H\"older continuity result following Proposition \ref{p:scaleiteration}.

We start by constructing the supersolution that \emph{sees a point}, which we will use to prove the improvement of oscillation in the first case. The supersolution is given explicitly and the proof that it is a supersolution amounts only to replace the values on the equation and check the inequality. It is a somewhat lengthy but elementary computation.

\begin{lemma} \label{l:supersolution}
For any $p>2$ and $A>0$, there is a constant $C$ (large) and $\eps_0>0$ (small), depending on $p$, $A$ and dimension only, such that the function
\[ U(x,t) = C t^{-\frac{1}{p-1}} (|x|^2 + t)^{p'/2} \]
is a supersolution of
\[ U_t + \frac 1A |DU|^p - \eps \ m^+ (D^2 U) \geq 0 \]
in $\R^d \times (0,+\infty)$, if $\eps < \eps_0$.
\end{lemma}

\begin{proof}
Just compute.
\begin{align*}
U_t &= C t^{-p'} (|x|^2 + t)^{p'/2} \left( -\frac{1}{p-1} + \frac{p'} 2 \frac{t}{|x|^2+t} \right) \\
\frac 1A |DU|^p &= \frac 1A C^p c_p t^{-p'} (|x|^2 + t)^{p'/2} \ \frac{|x|^p}{(|x|^2+t)^{p/2}} \\
\eps \  m^+ (D^2 U) &= C \eps  t^{\frac{-1}{p-1}} (|x|^2 + t)^{p'/2} \left(c'_p \frac{1}{|x|^2 + t} + c''_p \frac{|x|^2}{(|x|^2 + t)^2} \right)\\
& \leq C t^{\frac{-1}{p-1}} (|x|^2 + t)^{p'/2} \left( \eps c'''_p \frac{t}{|x|^2 + t} \right)
\end{align*}
Where $c_p$, $c'_p$, $c''_p$ and $c'''_p$ are constants depending only on $p$ and dimension. Therefore
\begin{align*}
U_t + \frac 1A |DU|^p - \eps \ m^+ (D^2 U) &\geq C t^{-p'} (|x|^2 + t)^{p'/2} \left( \frac{(\frac{p'} 2 - \eps c'''_p) t}{|x|^2+t} + \frac{ C^{p-1} c_p |x|^p}{A(|x|^2+t)^{p/2}} -\frac{1}{p-1} \right)
\end{align*}
all we are left to do is to show that if $C$ is large and $\eps$ is small enough (depending only on $p$ and dimension) then the last factor in the right hand side is nonnegative.

Let 
\begin{align*}
r &=  \frac{t}{|x|^2+t} \\
s &=  \frac{|x|^2}{|x|^2+t}
\end{align*}
So we have $r+s=1$ and moreover
\begin{align*}
U_t + \frac 1A |DU|^p - \eps \ m^+ (D^2 U) &\geq C t^{-p'} (|x|^2 + t)^{p'/2} \left( (\frac{p'} 2 - \eps c'''_p) r + \frac 1A C^{p-1} c_p s^{p/2} -\frac{1}{p-1} \right)
\end{align*}
We first choose $\eps_0$ small so that $(\frac{p'} 2 - \eps c'''_p) > 1/(p-1)$. Note that we can make this choice because $p>2$. Next, we choose $C$ large so that
\[ \frac 1A C^{p-1} c_p s^{p/2} - (\frac{p'} 2 - \eps c'''_p) s \geq -(\frac{p'} 2 - \eps c'''_p) + \frac{1}{p-1} \]
for all values of $s$. Then the desired inequality follows by writing $r = 1-s$.
\end{proof}

\begin{cor} \label{c:supersolution}
For any $\eta>0$ the function
\[ U(x,t) = C t^{-\frac{1}{p-1}} (|x|^2 + \eta t)^{p'/2} \]
is a supersolution of 
\[ U_t + \frac 1A |DU|^p - \eps \ m^+ (D^2 U) \geq 0 \]
if $\eps < \eta \eps_0$, where $C$ and $\eps_0$ are the same constants as in Lemma \ref{l:supersolution} (depending only on $A$, $p$ and dimension).
\end{cor}

\begin{proof}
Let $U$ be the function of Lemma \ref{l:supersolution}. For $r<1$, consider 
\[ U_r(x,t) = U(rx,r^pt) =C t^{-\frac{1}{p-1}} (|x|^2 + r^{p-2} t)^{p'/2} \]
which by scaling solves
\[ U_t + \frac 1A |DU|^p - \eps r^{p-2} \ m^+ (D^2 U) \geq 0 \]
so we choose $r$ small so that $r^{p-2} = \eta$ and $\eta \eps_0$ instead of the $\eps_0$ of Lemma \ref{l:supersolution}.
\end{proof}

The following Corollary reminds us to the Lax-Oleinik formula, except that it is not an identity but an inequality instead.

\begin{cor} \label{c:semi-LaxOleinik}
Let $R>0$ and $u$ be a subsolution of
\[ u_t + \frac 1A |Du|^p - \eps m^+ (D^2 u) - \eps \leq 0 \text{ in } B_{R+r} \times [0,1]\]
where $r = (1/C)^{1/p'}$ and $C$ and $\eps$ are as in Corollary \ref{c:supersolution}. Moreover, assume that $u \leq 1$ in the same cylinder. Then
\[ u(x,t) \leq \min_{y \in B_R} u(y,0) + U(x-y,t) + \eps t\]
where $U$ is the function given in Corollary \ref{c:supersolution}.
\end{cor}

\begin{proof}
The function 
\[ V(x,t) = \min_{y \in B_R} u(y,0) + U(x-y,t)+\eps t \]
is a minimum of translated supersolutions, then it is also a supersolution of
\[ V_t + \frac 1A |DV|^p - \eps m^+ (D^2 V) - \eps \geq 0 \]
in $\R^d \times [0,1]$. It is easy to check that $V \geq u$ on the boundary of the cylinder $B_{R+r} \times [0,1]$ since $V=u$ in $B_R \times \{0\}$, $V=+\infty$ in the rest of $\R^d \times \{0\}$ and $V \geq 1 \geq u$ on $\partial B_{R+r} \times [0,1]$.
\end{proof}

The following Lemma proves the first case of improvement of oscillation: when the function $u$ is small at one single point on $\{t=0\}$. In this case we will show that the maximum of $u$ is less than one in a smaller cylinder. We only need to assume that $u$ is a subsolution.

\begin{lemma} \label{l:upperhalf-improv}
Let $\theta \in (0,1/4)$. If $u$ is a subsolution of 
\[ u_t + \frac 1A |Du|^p - \eps m^+(D^2 u) - \eps \leq 0 \text{ in } B_{R+r} \times [0,1]\]
such that $u \leq \theta$ at some point in $B_R \times \{0\}$ and $0\leq  u \leq 1$ everywhere in that cylinder. If $\eps$ and $R$ are small enough (depending on $A$, $p$ and dimension but not on $\theta$), then $u \leq 1-\theta$ in $B_{R} \times [1/2,1]$.
\end{lemma}

\begin{proof}
Let $u(y_0,0) \leq \theta$. We apply the formula from Corollary \ref{c:semi-LaxOleinik}. 
\begin{align*}
u(x,t) &\leq \min_{y \in B_R} u(y,0) + U(x-y,t) + \eps t \\
&\leq \theta + U(x-y_0,t) + \eps \\
&\leq  \theta + C 2^{\frac{1}{p-1}} ((2R)^2 + \eta)^{p'/2}+ \eps
\end{align*}
Then we choose $\eps$ small so that $\eta$ is small, and $R$ small in order to obtain that the sum of the last two second terms is less than $1/2$, which is in turn less than $(1-2\theta)$.
\end{proof}

Now we construct the subsolution which we will use to prove the other half of the improvement of oscillation lemma.

\begin{lemma} \label{l:subsolution}
Let $b:\R \to \R$ be a smooth monotone decreasing function such that 
\[ b(t) = \begin{cases}
1 & \text{if } t < 3/4 \\ 
0 & \text{if } t > 1  
\end{cases} \]
If $C$ is large and $\theta$ is small enough depending $p$, $A$, $R$ and dimension,  the function
\[ L(x,t) =  \theta b\left(\frac{|x|}{R} + \frac{t}{4}\right) -  \frac{C \eps \theta^2}{R^2} t - \eps t\]
is a subsolution of
\[ L_t + A |DL|^p - \eps m^-(D^2 L) + \eps \leq 0\qquad {\rm in}\; \R^d\times (0,1) \]
\end{lemma}

\begin{proof}
We just compute
\begin{align*}
L_t  &=  (\theta/4) b'(|x|/R+t/4) -\frac{C \eps \theta^2}{R^2} - \eps \\
|DL| &= (\theta/R) |b'(|x|/R+t/4)| \\
m^- (D^2 L) &\geq - \frac{\theta^2}{R^2}( 2\|b'\|_\infty+\|b''\|_\infty/R)
\end{align*}
So, if we choose $C=  ( 2\|b'\|_\infty+\|b''\|_\infty/R)$  we obtain
\begin{align*}
L_t + A |DL|^p - \eps m^- (D^2L) + \eps &\leq (\theta/4) b' + \frac{A \theta^p}{R^p} |b'|^p \\
&\leq (\theta/4) |b'| (-1 + 4A \theta^{p-1}R^{-p} |b'|^{p-1})
\end{align*}
and the right hand side is non positive if we choose $\theta^{p-1}$ smaller than $R^p/(4 A \max|b'|^{p-1})$.
\end{proof}

Now we prove the other half of the improvement of oscillation. If the function $u$ is not small at any point on $\{t=0\}$ then it will not be small anywhere in a smaller cylinder. We only assume $u$ to be a supersolution here.

\begin{lemma} \label{l:lowerhalf-improv}
Let $u$ be a supersolution of 
\[ u_t + A |Du|^p - \eps m^-(D^2 u) + \eps \geq 0 \text{ in } B_{R} \times [0,1]\]
such that $u \geq \theta$ in $B_R \times \{0\}$ (for $\theta$ as in Lemma \ref{l:subsolution}) and $u \geq 0$ everywhere in that cylinder. If $\eps$ is small enough ($\eps < 1/(2c)$ for $c$ as in Lemma \ref{l:subsolution}) and $\eps < \theta/2$, then $u \geq \theta/2$ in $B_{R/2} \times [1/2,1]$.
\end{lemma}

\begin{proof}
We compare the function $u$ with the supersolution $L$ of Lemma \ref{l:subsolution}. We observe that $u \geq L$ in the parabolic boundary of the cylinder, therefore $u \geq L$ everywhere. Therefore, for $(x,t) \in B_{R/2} \times [1/2,1]$,
\begin{align*}
u(x,t) &\geq L(x,t) =  \theta b\left(\frac{|x|}{R} + \frac{t}{4}\right) -  \frac{c \eps \theta^2}{R^2} t - \eps t \\
&\geq \theta -  \frac{c \eps \theta^2}{2R^2}  - \frac{\eps}{2} \, \geq \; \frac{\theta}{2} 
\end{align*}for $\eps$ small enough. 
\end{proof}

Now we can combine Lemmas \ref{l:upperhalf-improv} and \ref{l:lowerhalf-improv} to obtain the full improvement of oscillation result.

\begin{lemma} \label{l:improvement-of-oscillation}
Assume that in $B_{R+r} \times [0,1]$ the function $u$ satisfies the two inequalities
\begin{align*}
u_t + A |Du|^p - \eps m^-(D^2 u) + \eps &\geq 0 \\
u_t + \frac 1 A |Du|^p - \eps m^+(D^2 u) - \eps &\leq 0  
\end{align*}
Assume also that the oscillation of $u$ in $B_{R+r} \times [0,1]$ is less or equal to $1$. If $\eps$ is small enough, then the oscillation of $u$ in $B_{R/2} \times [1/2,1]$ is less or equal to $1-\theta$, for constants $\theta$, $r$ and $R$ depending only on $A$, $p$ and dimension.
\end{lemma}

\begin{proof}
We start by defining
\[v(x,t) = u(x,t) - \min_{B_{R+r} \times [0,1]} u\]
so that $v$ satisfies the same equations as $u$ but $0 \leq v \leq 1$.

We choose the value of $r$ from Corollary \ref{c:semi-LaxOleinik} and the value of $R$ from Lemma \ref{l:upperhalf-improv} (which is independent of $\theta$). Then, for that value of $R$, we use the value of $\theta$ from  Lemma \ref{l:subsolution}. We choose $\eps$ small enough so that all lemmas apply.

If $v(y,0) \leq \theta$ for some point $y \in B_R$, we apply Lemma \ref{l:upperhalf-improv} and obtain $v \leq 1-\theta$ in $B_R \times [1/2,1]$. Otherwise we apply Lemma \ref{l:lowerhalf-improv} and obtain $v \geq \theta/2$ in $B_{R/2} \times [1/2,1]$.
\end{proof}

We rephrase Lemma \ref{l:improvement-of-oscillation} to a more convenient form.

\begin{cor} \label{c:oscillation-2ndorder}
Assume that in $B_1 \times [-1,0]$ the function $u$ satisfies the two inequalities
\begin{align}
u_t + A |Du|^p - \eps m^-(D^2 u) + \eps &\geq 0 \label{leq:ineq1}\\
u_t + \frac 1 A |Du|^p - \eps m^+(D^2 u) - \eps &\leq 0   \label{leq:ineq2}
\end{align}
Assume also $\osc_{B_1 \times [-1,0]} u \leq 1$. There is a $\lambda>0$ (depending only on $A$, $p$ and dimension), such that if $\eps$ is small enough, $\osc_{B_\lambda \times [-\lambda^{p/2},0]} u \leq 1$ is less or equal to $1-\theta$.
\end{cor}

\begin{proof} 
Just apply Lemma \ref{l:improvement-of-oscillation} to $u(x,t-1)$. Recall that the constants $R$ and $r$ were chosen small enough, so we can assume that $R+r < 1$ so that $B_{R+r} \subset B_1$ and we choose $\lambda = R/2$.
\end{proof}

We prove Theorem \ref{t:main2} following the procedure of Proposition \ref{p:scaleiteration}. This proof is identical to the proof of Theorem \ref{t:main1} in section \ref{sec:order1} except that we have to check the scaling of the second order term (for which we use the assumption $p>2$) and we apply the improvement of oscillation given by Corollary \ref{c:oscillation-2ndorder} instead of Corollary \ref{c:oscillation-1storder}.

\begin{proof}[Proof of Theorem \ref{t:main2}.]
Without loss of generality, we prove the estimate for $(x,t) = (0,0)$ and $(y,s)$ sufficiently close to $(0,0)$. Also without loss of generality we assume that $0 \leq u \leq 1$. Otherwise we consider $(u - \inf u)/(\osc u)$ instead, which also solves inequalities with \eqref{e:m1} and \eqref{e:m2} with a slightly modified $A$ (see section \ref{s:scaling}).

We start by scaling the equation so that the zeroth order term and the diffusion term are negligible. We consider $\tilde u(x,t) = u(\rho x, \rho^p t)$ which solves
\begin{align*}
\tilde u_t + A|D \tilde u|^p + \rho^p A - \rho^{p-2} A m^-(D^2 u) &\geq 0, \\ 
\tilde u_t + \frac 1A |D \tilde u|^p - \rho^p A - \rho^{p-2} A m^+(D^2 u) &\leq 0,
\end{align*}
We choose $\rho$ small enough so that $\rho^{p-2} A < \eps$, for the constant $\eps$ of Corollary \ref{c:oscillation-2ndorder}. Note that we can make the second order term negligible because $p>2$.

We now show that $\tilde u$ satisfies the iterative hypothesis of Proposition \ref{p:scaleiteration}. Let $\alpha\in(0, \frac{p}{2(p-1)})$ be a small positive constant that will be chosen below and set $Q_r=B_r(0)\times[-r^{p-\alpha(p-1)}, 0]$. Suppose that $\osc_{Q_r} \tilde u \leq r^\alpha$. Let $\tilde u^r(x,t) = r^{-\alpha} \tilde u(rx,r^{p-\alpha(p-1)}t)$. The function $\tilde u^r$ satisfies $\osc_{Q_1} \tilde u^r \leq 1$ and the inequalities
\begin{align*}
\tilde u^r_t + A|D \tilde u^r|^p + r^{p(1-\alpha)} \rho^p A - (r\rho)^{p-2} A m^-(D^2 \tilde u^r) &\geq 0, \\ 
\tilde u^r_t + \frac 1A |D \tilde u^r|^p - r^{p(1-\alpha)} \rho^p A  - (r\rho)^{p-2} A m^+(D^2 \tilde u^r) &\leq 0.
\end{align*}
Note that $r^{p(1-\alpha)} \rho^p A \leq \rho^p A < \eps$ and $(r\rho)^{p-2} A \leq \rho^{p-2} A < \eps$. Thus, we can apply Corollary \ref{c:oscillation-2ndorder} and obtain $\osc_{B_\lambda \times [-\lambda^{p/2},0]} \tilde u^r \leq 1-\theta$. Let us choose $\alpha$ such that $\lambda^\alpha \geq 1-\theta$. Therefore we have
\[ \osc_{Q_{\lambda r}} \tilde u = r^\alpha \osc_{B_\lambda \times [-\lambda^{p-\alpha(p-1)},0]} \tilde u^r \leq  r^\alpha \osc_{B_\lambda \times [-\lambda^{p/2},0]} \tilde u^r \leq (1-\theta) r^\alpha \leq (\lambda r)^\alpha. \]
And thus, we proceed with Proposition \ref{p:scaleiteration}. 
\end{proof}

\section{Second order equations with unbounded right hand side}\label{sec:unbounded}

In this section we prove Theorem \ref{t:intro3} about Hamilton-Jacobi equations with an unbounded right hand side.

Note that, with the assumptions of Theorem \ref{t:intro3}, $f$ is actually bounded from below by $-A$. An example of a situation in which the theorem is applicable is for hamiltonians $H$ of the form $H(x,t,q,X)=H_1(x,t,X)+H_2(x,t,q)$ with 
\[
H_1(x,t,X)= \sup_{a\in A} \left\{ -\frac12 {\rm Tr} \left(\sigma\sigma^*(x,t,a)X\right)\right\}
\; {\rm and} \; 
\frac 1 A |q|^p  - f(x,t) \leq H_2(x,t,q) \leq A |q|^p+ A
\]
where $\sup_{t,a} \|\sigma(\cdot,t,a)\|_{{\mathcal C}^1}$ is bounded  and $f$ is as above. Such Hamiltonian are often encountered in stochastic control. 

Theorem \ref{t:intro3} is a consequence of the following more precise statement.

\begin{thm}\label{t:mainUnbd}
Let $u$ be a continuous function which satisfies the following two inequalities in the viscosity sense
\begin{align}
u_t + A|Du|^p -A m^-(X) \geq & \; -A \qquad \text{and}\label{e:m1i}\\
u_t + \frac 1 A |Du|^p-{\rm Tr}(B(x,t)D^2u)  \leq   &\; f(x,t) \label{e:m2i}
\end{align}
in $B_1\times (-1,0]$. Then 
$$
|u(x,t)-u(y,s)|\leq C \left[ |x-y|^\alpha+|t-s|^{\frac{\alpha}{p-\alpha(p-1)}}\right] \qquad \forall (x,t),(y,s)\in B_{1/2}\times [-1/2,0]
$$
for some $\alpha \in (0, p')$ (where $1/p+1/p'=1$) and $C$ depending on $\sup_t\|B(\cdot,t)\|_{{\mathcal C}^1}$, $\|f\|_{L^m(B_1\times (0,1))}$, 
$A$, $p$, $\|u\|_\infty$ and dimension. 

\end{thm}


\begin{remark}
We were not able to find a simple proof of Theorem \ref{t:mainUnbd} using an explicit construction of sub and supersolutions, like in the previous sections for the case of bounded right-hand side. In the inviscid case ($B \equiv 0$) it would be possible to build a supersolution of the form
\[ \begin{split}
U(x,t) = C \min_{s<t, y} \bigg( &s^{-\frac 1 {p-1}} |y|^{p'} + (t-s)^{-\frac 1 {p-1}} |x-y|^{p'}  \\ &+ \int_0^s f\left( \frac r t y , r \right) \dd r + \int_s^t f\left( y + \frac {r-s} {t-s} (x-y) , r \right) \dd r \bigg),
\end{split} \]
which could be shown to be bounded by $||f||_{L^m}$, and then we could proceed like in section \ref{sec:order1}. However, it seems very difficult to find an estimate on the second derivative of such function $U(x,t)$, and therefore we need a different method in order to be able to handle equations with bounded right-hand side and second order diffusion. We are not able to handle diffusions as general as in the bounded coefficient case. In fact we currently do not know if the result of Theorem \ref{t:mainUnbd} would hold true if the inequality \eqref{e:m2i} was replaced by
\[u_t + \frac 1 A |Du|^p - A m^+(D^2u)  \leq f(x,t).\]
\end{remark}

In order to  show some estimates on subsolutions, we need an approximation argument. We use here ideas from \cite{Is95}, \cite{PLL83} adapted to our context. 

\begin{lemma}\label{l:approx} Let $u$ be a continuous viscosity subsolution of
\[
u_t + \frac{1}{A} |Du|^p-{\rm Tr}(B(x,t)D^2u)  \leq    f(x,t) \qquad {\rm in }\; B_R\times (a,b)
\]
Then, for any $\eps>0$, we can find two sequences of smooth maps $(u^n)$ and $(f^n)$, which converge uniformly in $B_{R-\eps}\times(a+\eps,b-\eps)$ to $u$ and $f$ respectively, such that 
\[
u^n_t + \frac{1}{2A} |Du^n|^p-{\rm div}(B(x,t)Du^n)  \leq    f_n(x,t)+ C \left[ \|B\|_{{\mathcal C}^1}^{p'}+\|B\|_{{\mathcal C}^1}\right] \qquad {\rm in }\; B_{R-\eps}\times(a+\eps,b-\eps)
\]
where $C$ only depends on $A$, $p$ and $d$. 
\end{lemma}

\begin{proof} Let us start as usual by regularizing $u$ by supconvolution: let $\delta>0$ small and 
\[
u^\delta(x,t)= \sup_{(y,s)} \left[u(y,s)-\frac{1}{2\delta}|(y,s)-(x,s)|^2\right]\;.
\]
Then it is well-known (see for instance Theorem 3 of \cite{Is95}) that $u^\delta$ satisfies 
\[
u^\delta_t + \frac 1 A |Du^\delta|^p-{\rm Tr}(B(x,t)D^2u^\delta)  \leq    f^\delta(x,t) \qquad {\rm in }\; B_{R-\delta}\times (a+C\delta,b-C\delta)
\]
for some $C=C(\|u\|_\infty)$ and for some $f^\delta$ which converge uniformly to $f$ as $\delta\to 0$. Recall that $u^\delta$ is semiconvex and therefore $Du^\delta$ has bounded variations. We denote by $D^2u^\delta=(D^2_{ij}u^\delta)$ the distributional derivative of $Du^\delta$. From Alexandrov Theorem, singular part $\nu^\delta= (\nu^\delta_{ij})$ of $D^2u^\delta$ is nonnegative while
its regular part   $\partial^2u^\delta=(\partial^2_{ij}u^\delta)$  satisfies $(u^\delta_t,Du^\delta(t,x), \partial^2u^\delta(t,x))\in {\mathcal P}^{2,+}u^\delta(x,t)$ a.e. (where ${\mathcal P}^{2,+}$ is the parabolic superjet of $u^\delta$, see \cite{CIL}). In particular
\begin{equation}\label{eq:subsolae}
u^\delta_t + \frac 1 A |Du^\delta|^p-{\rm Tr}(B(x,t)\partial^2u^\delta)  \leq    f^\delta(x,t) \qquad \mbox{\rm a.e. in }\; B_{R-\delta}\times (a+C\delta,b-C\delta)
\end{equation}

Next we consider a space-time nonnegative convolution kernel $\phi$ with support in $B_1$ and such that $|D \phi|^{p'}/\phi^{p'-1}$ is smooth (one can take for instance $\phi(x,t)= ce^{-1/(1-|(x,t)|^2)}$ if $|(x,t)|<1$ and $\phi(x,t)=0$ otherwise, $c$ being such that $\int \phi=1$).
We set $\phi_\gamma(t,x)= \frac{1}{\gamma^{d+1}}\phi((x,t)/\gamma)$ and denote by $u^{\delta,\gamma}= u^\delta\star \phi_\gamma$ the space-time convolution of $u^\delta$.  Integrating (\ref{eq:subsolae}) one gets
\[
\begin{array}{l}
\ds u^{\delta, \gamma}_t(x,t) + \frac 1 A \int \phi_\gamma(x-y,t-s) |Du^\delta(y,s)|^p\ dyds\\
\qquad \qquad \qquad  -\ds
\int \phi_\gamma(x-y,t-s) {\rm Tr}(B(y,s)\partial^2u^\delta(y,s))\ dyds  \leq    f^{\delta, \gamma}(x,t) \qquad \mbox{\rm in }\; 
B_{R'}\times (a',b')
\end{array}
\]
where $B_{R'}\times (a',b')=B_{R-\delta-\gamma}\times (a+\delta+\gamma,b-\delta-\gamma)$, $f^{\delta, \gamma}=f^{\delta}\star \phi_\gamma$. Now we note that, since $B$ can be written as $B=\sigma\sigma^*$, for some continuous 
square matrix $\sigma$, we have by positivity of $\nu$, 
\[ 
\sum_{i,j} \int \phi_\gamma B_{ij}\ d D^2_{ij}u^\delta= \sum_{i,j} \int (\sigma_i\sqrt{\phi_\gamma})(\sigma_j\sqrt{\phi_\gamma})d \nu_{ij}+
\sum_{i,j} \int \phi_\gamma B_{ij}\partial^2_{ij}u^\delta \geq \int \phi_\gamma {\rm Tr}(B\partial^2u^\delta)\;.
\]
On the other hand, by integration by parts, we get
\[
\sum_{i,j} \int \phi_\gamma B_{ij}d D^2_{ij}u^\delta= -\sum_{i,j} \int \frac{\partial B_{ij}}{\partial x_i}\phi_\gamma \frac{\partial u^\delta}{\partial x_j}
-\frac{\partial \phi_\gamma}{\partial x_i}  B_{ij}\frac{\partial u^\delta}{\partial x_j}
\]
Therefore
\begin{align*}
 -{\rm div}(BDu^{\delta,\gamma})+ \int \phi_\gamma {\rm Tr}(B\partial^2u^\delta)& \\
\qquad \leq 
-\sum_{i,j}  \frac{\partial B_{ij}}{\partial x_i} \int \phi_\gamma \frac{\partial u^\delta}{\partial x_j}
+B_{ij} \int \frac{\partial \phi_\gamma}{\partial x_i} \frac{\partial u^\delta}{\partial x_j}
&
+ \int \left( \frac{\partial B_{ij}}{\partial x_i}\phi_\gamma \frac{\partial u^\delta}{\partial x_j}
-\frac{\partial \phi_\gamma}{\partial x_i}  B_{ij}\frac{\partial u^\delta}{\partial x_j}\right)
\end{align*}
where, from our smoothness assumption on $|D\phi|^{p'}/\phi^{p'-1}= (|D\phi| \phi^{-1/p})^{p'}$,  
\begin{align*}
\left| B_{ij} \int \frac{\partial \phi_\gamma}{\partial x_i} \frac{\partial u^\delta}{\partial x_j}
-\int \frac{\partial \phi_\gamma}{\partial x_i}  B_{ij}\frac{\partial u^\delta}{\partial x_j}\right|
\leq & 
\int |B_{ij}(y,s)-B_{ij}(x,t)| |\frac{\partial \phi_\gamma}{\partial x_i} | \phi_\gamma^{-1/p}
\phi_\gamma^{1/p} |\frac{\partial u^\delta}{\partial x_j}| \\
\leq & C\|B\|_{{\mathcal C}^1} \left[ \int \phi_\gamma |Du^\delta|^{p}\right]^{1/p} 
\end{align*}
while 
\[
\left|\sum_{i,j}  \frac{\partial B_{ij}}{\partial x_i} \int \phi_\gamma \frac{\partial u^\delta}{\partial x_j}
+ \int \frac{\partial B_{ij}}{\partial x_i}\phi_\gamma \frac{\partial u^\delta}{\partial x_j}\right|
\leq C \|B\|_{{\mathcal C}^1} \int \phi_\gamma |Du^\delta|\;.
\]
So
\begin{align*}
-\frac{1}{2A}\int \phi_\gamma |Du^\delta|^p  -{\rm div}(BDu^{\delta,\gamma})+ \int \phi_\gamma {\rm Tr}(B\partial^2u^\delta)
\leq C\left[ \|B\|_{{\mathcal C}^1}^{p'}+\|B\|_{{\mathcal C}^1}\right]\;.
\end{align*}
Putting our estimates together, we get
\[
u^{\delta, \gamma}_t + \frac{1}{2A} \int \phi_\gamma |Du^\delta|^p -{\rm div}(BDu^{\delta,\gamma}) \leq    f^{\delta, \gamma}(x,t) 
+C\left[ \|B\|_{{\mathcal C}^1}^{p'}+\|B\|_{{\mathcal C}^1}\right]
\qquad \mbox{\rm in } B_{R'}\times (a',b')
\]
where 
\[
\frac{1}{2A} \int \phi_\gamma |Du^\delta|^p \geq 
\frac{1}{2A}  |Du^{\delta,\gamma}|^p\;.
\]
Choosing $\delta$  and   $\gamma$ sufficiently small then gives the result. 
\end{proof}

We now have the first step of the improvement of the oscillation. 

\begin{lemma} Let $\theta\in (0,1/4)$. Assume that $u$ is a continuous subsolution to 
\[ u_t +  |D u|^p -  {\rm Tr}(B(x,t) D^2u)  \leq f(x,t) \text{ in } B_{R+1} \times (0,1)\]
with $0\leq u \leq 1$ and $u\leq \theta$ at some point in $B_R\times \{0\}$. If $\eps$, $\|B\|_{{\mathcal C}^1}$ and 
 $\displaystyle{  \|f\|_{L^m(B_{R+1}\times [0,1])}  }$
are sufficiently small (depending on $p$ and dimension but not on $\theta$), then $u\leq 1-\theta$
in $B_{R}\times [1/2,1]$. 
\end{lemma}

\begin{proof}  Using Lemma \ref{l:approx} it is enough to prove the result for a smooth map $u$ which satisfies 
 \[
u_t + \frac{1}{2A} |Du|^p-{\rm div}(B(x,t)Du)  \leq    f(x,t)+ C\left[ \|B\|_{{\mathcal C}^1}^{p'}+\|B\|_{{\mathcal C}^1}\right] \qquad {\rm in }\; B_{1+R}\times (0,1)
\]
We set $g=f+C \left[ \|B\|_{{\mathcal C}^1}^{p'}+\|B\|_{{\mathcal C}^1}\right]$.  

Let $y\in B_R$ be such that $u(y,0)\leq \theta$ and let us fix $(x,t)\in B_R\times [1/2,1]$. Let also $\phi=\phi(z)$ be a smooth, nonnegative convolution kernel, with support in the unit ball. We assume as before that $|D \phi|^{p'}/\phi^{p'-1}$ is smooth. Since $p>2$ and $d< p(m-1)$, we can fix some $\beta\in (1/p, (1/p')\wedge (m-1)/d)$. For $\delta>0$ to be chosen later we set 
\[
\varphi(z,t)= \frac{1}{(\xi_\delta(s))^d} \phi\left(\frac{(1-s/t)y+(s/t)x +z}{\xi_\delta(s)}\right)\qquad {\rm where}\qquad
\xi_\delta(s)= \delta \min\{s^\beta, (t-s)^\beta\}. 
\]
Then we consider 
$$
\rho(s)= \int_{\R^d} \varphi(z,s)u(z,s)dz= \int_{\R^d} \phi(z) u\left((1-s/t)y+(s/t)x +\xi_\delta(s)z,s\right)dz \;.
$$
We note that $\rho$ is absolutely continuous, $\rho(0)=u(y,0)$ and $\rho(t)= u(x,t)$. Then
$$
\rho_t(s)=  \int_{\R^d} \phi (u_t+ Du. (\frac{x-y}{t}+ \xi_\delta'z))dz
$$
where in the above equality and the (in)equalities below, $\phi=\phi(z)$ and $u$ (as well as $B$ and $g$) and its derivative have for argument 
$((1-s/t)y+(s/t)x +\xi_\delta(s)z,s)$. Hence
$$
\rho_t(s)=  \int_{\R^d} \phi (-|D u|^p + g+ Du. (\frac{x-y}{t}+ \xi_\delta'z)) -\frac{(B^*D \phi).D u}{\xi_\delta} 
$$
where
$$
\phi (-|D u|^p +  Du. (\frac{x-y}{t}+ \xi_\delta'z)) -\frac{(B^*D \phi).D u}{\xi_\delta} \leq 
C( \phi (\frac{|x-y|^{p'}}{t^{p'}}+  |\xi_\delta'|^{p'}|z|^{p'}+ \|B\|_\infty^{p'}\frac{|D \phi|^{p'}}{\phi^{p'-1}}|\xi_\delta|^{-p'})
$$
So, if we set $\eps= \|B\|_\infty$, 
$$
u(x,t)\leq u(y,0) +C|x-y|^{p'}t^{1-p'}+ \int_0^t\int_{\R^d} \phi g + C \int_0^t |\xi_\delta'|^{p'}+ \eps^{p'}|\xi_\delta|^{-p'}\;.
$$
Let us estimate the right-hand side:
$$
\begin{array}{rl}
\int_0^t\int_{\R^d} \phi g = \int_0^t\int_{\R^d} \varphi(z,s) g(z,s)dzds \; \leq &
 \|g\|_m \left(\int_0^t \int_{B_1} \xi_\delta^{d(1-m')}\phi^{m'}\right)^{1/m'}\\
 \leq &  C  \|g\|_m \left(2\int_0^{t/2} \delta^{d(1-m')} s^{\beta d(1-m')}ds \right)^{1/m'}\; \leq \;  
  C \|g\|_m\delta^{-d/m}
  \end{array}
$$
(where $1/m+1/m'=1$, since $\beta d(1-m')>-1$ because $\beta<(m-1)/d$) and
$$
\int_0^t |\xi_\delta'|^{p'}+ \eps^{p'}|\xi_\delta|^{-p'} = 2\int_0^{t/2} \delta^{p'} \beta^{p'} s^{(\beta-1) p'} +\eps^{p'}\delta^{-p'} s^{-\beta p'}\ ds
\leq C(\delta^{p'}+ \eps^{p'}\delta^{-p'})
$$
(since $(\beta-1) p'>-1$ because $\beta>1/p$ and $\beta<1/p'$). 
Therefore, assuming that $\|B\|_\infty<\eps_0$ and $\|g\|_m<\eps_0$ for some $\eps_0\in(0,1)$ to be fixed below, and
 choosing $\delta= \min\{\eps_0^{1/2}, \eps_0^{1/(p'+d/m)}\}$ one gets
 $$
 u(x,t)\leq u(y,0) +CR^{p'}(1/2)^{1-p'}+ C\left[ \eps_0^{p'/2}+ \eps_0^{p'/(p'+d/m)}\right]
$$
We can now fix $R$ and $\eps_0$ sufficiently small so that the remaining term is less than $1/2$: in this case
$$
u(x,t)\leq \theta+1/2\leq 1-\theta\;.
$$
\end{proof}

Since the control of supersolutions remains unchanged, one can conclude exactly as in the case of bounded right-hand side that the improvement of oscillations holds:
\begin{cor} \label{c:oscillation-unbd}
Assume that in $B_1 \times [-1,0]$ the function $u$ satisfies the two inequalities
\begin{align}
u_t + A|Du|^p -\eps m^-(X) \geq & \; -\eps \label{eq:ineq1f}\\
u_t + \frac 1 A |Du|^p-{\rm Tr}(B(x,t)D^2u)  \leq   &\; f(x,t) \label{eq:ineq2f}
\end{align}
Assume also $\osc_{B_1 \times [-1,0]} u \leq 1$. There is a $\lambda>0$ (depending only on $A$, $p$ and dimension), such that if $\eps$ is small enough and  if
$\|B\|_{{\mathcal C}^1}<\eps$, $\|g\|_{m,R+1}<\eps$, then $\osc_{B_\lambda \times [-\lambda^{p/2},0]} u \leq 1$ is less or equal to $1-\theta$.
\end{cor}

The proof of Theorem \ref{t:mainUnbd} is identical to the proof of Theorem \ref{t:main1} or Theorem \ref{t:main2} but using Corollary \ref{c:oscillation-unbd} instead of Corollary \ref{c:oscillation-1storder} or \ref{c:oscillation-2ndorder}, and verifying that the assumptions of Corollary \ref{c:oscillation-unbd} hold at small scales. This is the point where the assumption $m > d/p+1$ plays a role.

\begin{proof}[Proof of Theorem \ref{t:mainUnbd}.] 
Without loss of generality, we prove the estimate for $(x,t) = (0,0)$ and $(y,s)$ sufficiently close to $(0,0)$. Also without loss of generality we assume that $0 \leq u \leq 1$.

We start by scaling the equation so that the zeroth order term and the diffusion term are negligible. We consider $\tilde u(x,t) = u(\rho x, \rho^p t)$ which solves
\begin{align*}
\tilde u_t + A|D \tilde u|^p  - \rho^{p-2} A m^-(D^2 u) &\geq -\rho^p A, \\ 
\tilde u_t + \frac 1A |D \tilde u|^p - \rho^{p-2} \tr(B(\rho x,\rho^p t) D^2 u) &\leq \rho^p f(\rho x,\rho^p t),
\end{align*}
We choose $\rho$ small enough so that $\rho^{p-2} A < \eps$, for the constant $\eps$ of Corollary \ref{c:oscillation-unbd} and also
\begin{align*}
\|\rho^p f(\rho \cdot,\rho^p \cdot)\|_{L^m(B_1\times [-1,0])}\leq \rho^{p(1-1/m)-d/m}\|f\|_{L^m(B_1\times [-1,0])} &\leq \eps \qquad \text{and} \\
\| \rho^{p-2}  B(\rho \cdot ,\rho \cdot)\|_{{\mathcal C}^1} &\leq \eps.
\end{align*}

Note that $|\grad_x \rho^{p-2} B(\rho \cdot ,\rho \cdot)| = \rho^{p-1} |\grad_x B|$.

We now show that $\tilde u$ satisfies the iterative hypothesis of Proposition \ref{p:scaleiteration}. Let
$Q_r=B_r(0)\times[-r^{p-\alpha(p-1)}, 0]$ and suppose that $\osc_{Q_r(x,t)} \tilde u \leq r^\alpha$, where $\alpha$ is a small positive constant that will be chosen below. Let $\tilde u^r(x,t) = r^{-\alpha} \tilde u^r(rx,r^{p-\alpha(p-1)})$. The function $\tilde u^r$ satisfies $\osc_{Q_1} \tilde u^r \leq 1$ and the inequalities
\begin{align*}
\tilde u^r_t + A|D \tilde u^r|^p - (r\rho)^{p-2} A m^-(D^2 \tilde u^r) &\geq -r^{p(1-\alpha)} \rho^p A, \\ 
\tilde u^r_t + \frac 1A |D \tilde u^r|^p - (r\rho)^{p-2} \tr(B(r\rho x,(r\rho)^p t) D^2 u) &\leq r^{p(1-\alpha)} \rho^p f(r\rho x,(r\rho)^p t) =: g_r(x,t).
\end{align*}
Note that
\[ ||g_r||_{L^m} = ||r^{p(1-\alpha)} \rho^p f(r\rho \cdot,(r\rho)^p t)||_{L^m} \leq r^{p(1-\alpha-1/m)-d/m} \rho^{p(1-1/m) - d/m} ||f||_{L^m}. \]
Therefore $||g_r||_{L^m}$ stays bounded for $r<1$ as long as $\alpha$ is chosen small enough so that $p(1-\alpha-1/m)-d/m \geq 0$, which is possible since by assumption $p(1-1/m)-d/m > 0$.

Thus, we can apply Corollary \ref{c:oscillation-unbd} and obtain $\osc_{B_\lambda \times [-\lambda^{p/2},0]} \tilde u^r \leq 1-\theta$. Let us choose $\alpha$ also smaller than $1/2$ and such that $\lambda^\alpha \geq 1-\theta$. Therefore we have
\[ \osc_{\lambda r} \tilde u = r^\alpha \osc_{B_\lambda \times [-\lambda^{p-\alpha(p-1)},0]} \tilde u^r \leq  r^\alpha \osc_{B_\lambda \times [-\lambda^{p/2},0]} \tilde u^r \leq (1-\theta) r^\alpha \leq (\lambda r)^\alpha. \]
And thus, we proceed with Proposition \ref{p:scaleiteration}. 
\end{proof}

\bibliographystyle{plain}

\end{document}